\documentclass[12pt]{amsart}
\usepackage[english]{babel}
\usepackage{lipsum}
\usepackage{mathrsfs}
\usepackage{stix}
\usepackage{fullpage}
\usepackage{mathtools}
\usepackage{amsmath}
\usepackage{tikz-cd}
\usepackage[colorlinks]{hyperref}
\usepackage{fullpage}
\usepackage{mathtools}
\usepackage{amsmath}

\newtheorem{thm}{Theorem}[section]
\newtheorem{lem}[thm]{Lemma}

\newtheorem{assum}[thm]{Assumption}

\theoremstyle{definition}

\theoremstyle{remark}
\newtheorem{rem}{Remark}[section]
\newtheorem{defn}{Definition}

\numberwithin{equation}{section}

\begin{document}

\title[Heat equation with singular potentials]{The heat equation with strongly singular potentials}

\author[A. Altybay]{Arshyn Altybay$^{1,2,3,A}$}
\author[M. Ruzhansky]{Michael Ruzhansky$^{1,4,B}$}
\author[M. Sebih]{Mohammed Elamine Sebih$^{1,5,C}$}
\author[N. Tokmagambetov]{Niyaz Tokmagambetov$^{1,2,D}$}
\address{$^{1}$Department of Mathematics: Analysis, Logic and Discrete Mathematics, Ghent University, Belgium; $^{2}$Al-Farabi Kazakh National University, Almaty, Kazakhstan; $^{3}$Institute of Mathematics and Mathematical Modeling, Almaty, Kazakhstan; $^{4}$School of Mathematical Sciences, Queen Mary University of London, United Kingdom; $^{5}$Laboratory of Analysis and Control of Partial Differential Equations, Djillali Liabes University, Sidi Bel Abbes, Algeria}
\address{{\it E-mails:} $^{A}${\rm arshyn.altybay@gmail.com}, $^{B}${\rm michael.ruzhansky@ugent.be}, $^{C}${\rm sebihmed@gmail.com}, $^{D}${\rm tokmagambetov@math.kz}}



\begin{abstract}
In this paper we consider the heat equation with strongly singular potentials and prove that it has a "very weak solution". Moreover, we show the uniqueness and consistency results in some appropriate sense. The cases of positive and negative potentials are studied. Numerical simulations are done: one suggests so-called "laser heating and cooling" effects depending on a sign of the potential. The latter is justified by the physical observations.
\end{abstract}

\maketitle


\section{Introduction}
After the pioneering works due to Baras and Goldstein \cite{BG84a}, \cite{BG84b}, the heat equation with inverse-square potential in bounded and unbounded domains has attracted considerable attention during the last decades, we cite \cite{AFP17}, \cite{FM15}, \cite{Gul02}, \cite{IKM19}, \cite{IO19}, \cite{Mar03}, \cite{MS10} and \cite{VZ00} to name only few.

Our aim is to contribute to the study of the heat equation by incorporating more singular potentials. The major obstacle for considering general coefficients is related to the multiplication problem for distributions \cite{Sch54}. There are several ways to overcome this problem. One way is to use the notion of very weak solutions.

The concept of very weak solutions was introduced in \cite{GR15} for the analysis of second order hyperbolic equations with non-regular time-dependent coefficients, and was applied for the study of several physical models in \cite{MRT19}, \cite{RT17a}, and in \cite{RT17b}. In these papers the very weak solutions are presented for equations with time-dependent coefficients. In the recent paper \cite{Gar20}, the author introduces the concept of the very weak solution for the wave equation with space-depending coefficient. Here we study the Cauchy problem for the heat equation with a non-negative potential, we allow the potential to be discontinuous or even less regular and we want to apply the concept of very weak solutions to establish a well-posedness result. Also, we note that very weak solutions for fractional Klein-Gordon equations with singular masses were considered in \cite{ARST21}.

In this paper we consider the heat equation with strongly singular potentials, in particular, with a $\delta$-function and with a behaviour like "multiplication" of $\delta$-functions. The existence result of very weak solutions is proved. Also, we show the uniqueness of the very weak solution and the consistency with the classical solution in some appropriate senses. The cases of positive and negative potentials are studied and numerical simulations are given. Finally, one observes so-called "laser heating and cooling" effects depending on a sign of the potential.

\section{Part I: Non-negative potential}

In this section we consider the case when the potential $q$ is non-negative. But first let us fix some notations. For our convenience, we will write $f\lesssim g$, which means that there exists a positive constant $C$ such that $f \leq Cg$. Also, let us define
\begin{equation*}
\Vert u(t,\cdot)\Vert_{k}:= \Vert \nabla u(t,\cdot)\Vert_{L^2} + \sum_{l=0}^{k}\Vert \partial_{t}^{l}u(t,\cdot)\Vert_{L^2},
\end{equation*}
for all $k\in\mathbb Z_{+}$. In the case when $k=0$, we simply use $\Vert u(t,\cdot)\Vert$ instead of $\Vert u(t,\cdot)\Vert_{0}$.

Fix $T>0$. In the domain $\Omega:=\left(0,T\right)\times \mathbb{R}^{d}$ we consider the heat equation
\begin{equation}
\label{Equation}
\partial_{t}u(t,x)-\Delta u(t,x) + q(x)u(t,x)=0, \,\, (t,x)\in\Omega,
\end{equation}
with the Cauchy data $u(0,x)=u_{0}(x),$ where the potential $q$ is assumed to be non-negative and singular.

In the case when the potential is a regular function, we have the following lemma.
\begin{lem}\label{Lemma 1}
Let $u_{0}\in H^{1}(\mathbb{R}^d)$ and suppose that $q\in L^{\infty}(\mathbb{R}^d)$ is non-negative. Then, there is a unique solution $u\in C^{1}(\left[0,T\right]; L^{2}) \cap C(\left[0,T\right]; H^{1})$ to (\ref{Equation}) and it satisfies the energy estimate
\begin{equation}
\Vert u(t,\cdot)\Vert \lesssim \left(1+\Vert q\Vert_{L^{\infty}}\right)\Vert u_{0}\Vert_{H^{1}}. \label{Energy estimate}
\end{equation}
\end{lem}

\begin{proof}
By multiplying the equation (\ref{Equation}) by $u_t$ and integrating with respect to $x$, we obtain
\begin{equation}
Re \left(\langle u_{t}(t,\cdot),u_{t}(t,\cdot)\rangle_{L^2} + \langle -\Delta u(t,\cdot),u_{t}(t,\cdot)\rangle_{L^2} + \langle q(\cdot)u(t,\cdot),u_{t}(t,\cdot)\rangle_{L^2} \right)=0. \label{Energy functional}
\end{equation}
One observes
\begin{equation*}
Re \langle u_{t}(t,\cdot),u_{t}(t,\cdot)\rangle_{L^2}=\langle u_{t}(t,\cdot),u_{t}(t,\cdot)\rangle_{L^2}=\Vert u_{t}(t,\cdot)\Vert_{L^2}^{2}.
\end{equation*}
Also, we see that
\begin{equation*}
Re \langle -\Delta u(t,\cdot),u_{t}(t,\cdot)\rangle_{L^2}=\frac{1}{2}\partial_{t}\langle \nabla u(t,\cdot),\nabla u(t,\cdot)\rangle_{L^2}=\frac{1}{2}\partial_{t}\Vert \nabla u(t,\cdot)\Vert_{L^2}^{2}
\end{equation*}
and
\begin{equation*}
Re \langle q(\cdot)u(t,\cdot),u_{t}(t,\cdot)\rangle_{L^2}=\frac{1}{2}\partial_{t}\langle q^{\frac{1}{2}}(\cdot)u(t,\cdot),q^{\frac{1}{2}}(\cdot)u(t,\cdot)\rangle_{L^2}=\frac{1}{2}\partial_{t}\Vert q^{\frac{1}{2}}(\cdot) u(t,\cdot)\Vert_{L^2}^{2}.
\end{equation*}
It follows from (\ref{Energy functional}) that
\begin{equation}
\partial_{t}\left[\Vert \nabla u(t,\cdot)\Vert_{L^2}^{2}+\Vert q^{\frac{1}{2}}(\cdot) u(t,\cdot)\Vert_{L^2}^{2}\right] =-2 \Vert u_{t}(t,\cdot)\Vert_{L^2}^{2}. \label{Energy functional 1}
\end{equation}
Let us denote by
\begin{equation*}
E(t):=\Vert \nabla u(t,\cdot)\Vert_{L^2}^{2}+\Vert q^{\frac{1}{2}}(\cdot) u(t,\cdot)\Vert_{L^2}^{2},
\end{equation*}
the energy functional. 
It follows from (\ref{Energy functional 1}) that $E^{\prime}(t) \leq 0$, and thus
\begin{equation*}
E(t)\leq E(0).
\end{equation*}
By taking into account that
$\Vert q^{\frac{1}{2}}(\cdot)u_{0}(\cdot)\Vert_{L^2}^{2}$ can be estimated by
\begin{equation*}
\Vert q^{\frac{1}{2}}(\cdot)u_{0}(\cdot)\Vert_{L^2}^{2}\leq \Vert q(\cdot)\Vert_{L^{\infty}}\Vert u_{0}(\cdot)\Vert_{L^2}^{2},
\end{equation*}
we get
\begin{equation*}
\Vert \nabla u(t,\cdot)\Vert_{L^2}^{2}+\Vert q^{\frac{1}{2}}(\cdot) u(t,\cdot)\Vert_{L^2}^{2} \leq \Vert \nabla u_{0}\Vert_{L^2}^{2}+\Vert q(\cdot)\Vert_{L^{\infty}}\Vert u_{0}\Vert_{L^2}^{2}.
\end{equation*}
Thus, we have
\begin{equation}
\Vert q^{\frac{1}{2}}(\cdot) u(t,\cdot)\Vert_{L^2}^{2}\leq \Vert \nabla u_{0}\Vert_{L^2}^{2}+\Vert q(\cdot)\Vert_{L^{\infty}}\Vert u_{0}\Vert_{L^2}^{2} \label{Estimate qxu}
\end{equation}
and
\begin{equation*}
\Vert \nabla u(t,\cdot)\Vert_{L^2}^{2}\leq \Vert \nabla u_{0}\Vert_{L^2}^{2}+\Vert q(\cdot)\Vert_{L^{\infty}}\Vert u_{0}\Vert_{L^2}^{2},
\end{equation*}
and consequently, one can be seen that
\begin{equation}
\Vert \nabla u(t,\cdot)\Vert_{L^2}\leq \left(1+\Vert q\Vert_{L^{\infty}}^{\frac{1}{2}}\right)^{2}\Vert u_{0}\Vert_{H^{1}}. \label{Estimate grad u}
\end{equation}

To obtain the estimate for $u$, we rewrite the equation (\ref{Equation}) as follows
\begin{equation}
\label{Equation Duhamel}
u_{t}(t,x)-\Delta u(t,x)= -q(x)u(t,x) ,~~~(t,x)\in\left(0,T\right)\times \mathbb{R}^{d}.
\end{equation}
Here, considering $-q(x)u(t,x)$ as a source term, we denote it by $f(t, x):=-q(x)u(t,x)$. By using Duhamel's principle (see, e.g. \cite{Eva98}), we represent the solution to (\ref{Equation Duhamel}) in the form
\begin{equation}
u(t,x)=\phi_{t}\ast u_{0}(x) + \int_{0}^{t}\phi_{t-s}\ast f_{s}(x)ds, \label{Sol Eq Duhamel}
\end{equation}
where $f_{s}=f(s, \cdot)$ and $\phi_{t}=\phi(t, \cdot)$. Here, $\phi$ is the fundamental solution (heat kernel) to the heat equation, and it satisfies
\begin{equation*}
\Vert \phi(t, \cdot)\Vert_{L^{1}}=1.
\end{equation*}
Now, taking the $L^{2}$-norm in (\ref{Sol Eq Duhamel}) and using Young's inequality, we arrive at
\begin{align*}
\Vert u(t,\cdot)\Vert_{L^{2}} & \leq \Vert \phi_{t}\Vert_{L^{1}}\Vert u_{0}\Vert_{L^{2}} + \int_{0}^{T}\Vert \phi_{t-s}\Vert_{L^{1}}\Vert f_{s}\Vert_{L^{2}} ds\\
& \leq \Vert u_{0}\Vert_{L^{2}} + \int_{0}^{T}\Vert f_{s}\Vert_{L^{2}} ds\\
& \leq \Vert u_{0}\Vert_{L^{2}} + \int_{0}^{T}\Vert q(\cdot)u(s,\cdot)\Vert_{L^{2}} ds.
\end{align*}
We estimate the term $\Vert q(\cdot)u(s,\cdot)\Vert_{L^{2}}$ as
\begin{equation*}
\Vert q(\cdot)u(s,\cdot)\Vert_{L^{2}} \leq \Vert q\Vert_{L^{\infty}}^{\frac{1}{2}}\Vert q^{\frac{1}{2}}u(s,\cdot)\Vert_{L^{2}},
\end{equation*}
and using the estimate (\ref{Estimate qxu}), one observes
\begin{equation}
\Vert u(t,\cdot)\Vert_{L^2}\lesssim \left(1+\Vert q\Vert_{L^{\infty}}^{\frac{1}{2}}\right)^{2}\Vert u_{0}\Vert_{H^{1}}. \label{Estimate u}
\end{equation}
Summing the estimates proved above, we conclude (\ref{Energy estimate}).

\end{proof}

\begin{rem}
We can also prove that the estimate
\begin{equation*}
\Vert \partial_{t}^{k}u(t,\cdot)\Vert_{L^2}\lesssim \left(1+\Vert q\Vert_{L^{\infty}}\right)
\Vert u_{0}\Vert_{H^{2k+1}},
\end{equation*}
is valid for all $k\geq 0$, by requiring higher regularity on $u_0$. To do so, we denote by $v_{0}:=u$ and its derivatives by $v_{k}:=\partial_{t}^{k}u$, where $u$ is the solution of the Cauchy problem (\ref{Equation}). Using (\ref{Estimate u}) and the property that if $v_{k}$ solves the equation
\begin{equation*}
\partial_{t}v_{k}(t,x)-\Delta v_{k}(t,x) + q(x)v_{k}(t,x)=0,
\end{equation*}
with the initial data $v_{k}(0,x)$, then $v_{k+1}=\partial_{t}v_{k}$ solves the same equation with the initial data
\begin{equation*}
v_{k+1}(0,x)=\Delta v_{k}(0,x)-q(x)v_{k}(0,x),
\end{equation*}
we get our estimate for $\partial_{t}^{k}u$ for all $k\geq 0$.
\end{rem}

To prove the uniqueness and consistency of the very weak solution, we will also need the following lemma.
\begin{lem}
\label{Lemma 2}
Let $u_{0}\in H^{1}(\mathbb{R}^{d})$ and assume that $q\in L^{\infty}(\mathbb{R}^d)$ is non-negative. Then, the estimate
\begin{equation}
\label{Energy estimate 2}
\Vert u(t,\cdot)\Vert_{L^2} \lesssim \Vert u_{0}\Vert_{L^2},
\end{equation}
holds for the unique solution $u\in C^{1}(\left[0,T\right];L^{2})\cap C(\left[0,T\right];H^{1})$ of the Cauchy problem (\ref{Equation}).
\end{lem}

\begin{proof}
Again, by multiplying the equation (\ref{Equation}) by $u$ and integrating over $\mathbb{R}^{d}$ in $x$, we derive
\begin{equation*}
Re \left(\langle u_{t}(t,\cdot),u(t,\cdot)\rangle_{L^2} + \langle -\Delta u(t,\cdot),u(t,\cdot)\rangle_{L^2} + \langle q(\cdot)u(t,\cdot),u(t,\cdot)\rangle_{L^2} \right)=0.
\end{equation*}
Using the similar arguments as in Lemma \ref{Lemma 1}, we obtain
\begin{equation}
\label{Energy functional 2}
\partial_{t}\Vert u(t,\cdot)\Vert_{L^2}^{2} = - \Vert \nabla u(t,\cdot)\Vert_{L^2}^{2} - \Vert q^{\frac{1}{2}}(\cdot)u(t,\cdot)\Vert_{L^2}^{2} \leq0.
\end{equation}
This ends the proof of the lemma.
\end{proof}

Now, let us show that the Cauchy problem (\ref{Equation}) has a very weak solution. We start by regularising the coefficient $q$ and the initial data $u_0$ using a suitable mollifier $\psi$, generating families of smooth functions $(q_{\varepsilon})_{\varepsilon}$ and $(u_{0,\varepsilon})_{\varepsilon}$. Namely,
\begin{equation*}
q_{\varepsilon}(x)=q\ast \psi_{\varepsilon }(x),~~~~u_{0,\varepsilon}(x)=u_{0}\ast \psi_{\varepsilon }(x),
\end{equation*}
where
\begin{equation*}
\psi_{\varepsilon }(x)=\omega(\varepsilon)^{-d}\psi(x/\omega(\varepsilon)), ~~~\varepsilon\in\left(0,1\right],
\end{equation*}
and $\omega(\varepsilon)$ is a positive function converging to $0$ as $\varepsilon \rightarrow 0$ to be chosen later. The function $\psi$ is a Friedrichs-mollifier, i.e. $\psi\in C_{0}^{\infty}(\mathbb{R}^{d})$, $\psi\geq 0$ and $\int\psi =1$.

\begin{assum}
On the regularisation of the coefficient $q$ and the initial data $u_{0}$ we make the following assumptions:
there exist $N, N_{0}\in \mathbb{N}_{0}$ such that
\begin{equation}
\label{Moderetness hyp data}
\Vert u_{0,\varepsilon}\Vert_{H^1}\leq C_{0}\omega(\varepsilon)^{-N_0},
\end{equation}
and
\begin{equation}
\label{Moderetness hyp coeff}
\Vert q_{\varepsilon}\Vert_{L^{\infty}}\leq C\omega(\varepsilon)^{-N},
\end{equation}
for $\varepsilon\in(0, 1]$.
\end{assum}

\begin{rem}
We note that such assumptions are natural for distributions. Indeed, by the structure theorems for distributions (see, e.g. \cite{FJ98}), we know that every compactly supported distribution can be represented by a
finite sum of (distributional) derivatives of continuous functions. Precisely, for $T\in \mathcal{E}'(\mathbb{R}^{d})$ we can find $n\in \mathbb{N}$ and functions $f_{\alpha}\in C(\mathbb{R}^{d})$ such that $T=\sum_{\vert \alpha\vert \leq n}\partial^{\alpha}f_{\alpha}$. The convolution of $T$ with a mollifier yields
\begin{equation*}
T\ast\psi_{\varepsilon}=\sum_{\vert \alpha\vert \leq n}\partial^{\alpha}f_{\alpha}\ast\psi_{\varepsilon}=\sum_{\vert \alpha\vert \leq n}f_{\alpha}\ast\partial^{\alpha}\psi_{\varepsilon}=\sum_{\vert \alpha\vert \leq n}\omega(\varepsilon)^{-\vert\alpha\vert}f_{\alpha}\ast\left(\omega(\varepsilon)^{-1}\partial^{\alpha}\psi(x/\omega(\varepsilon))\right).
\end{equation*}
It is clear that $T$ satisfies the above assumptions.
\end{rem}

\subsection{Existence of very weak solutions}

In this subsection we deal with the existence of very weak solutions. We start by calling the definition of the moderateness.
\begin{defn}[Moderateness] \label{Def:Moderetness}
Let $X$ be a Banach space with the norm $\|\cdot\|_{X}$. Then we say that a net of functions $(f_{\varepsilon})_{\varepsilon}$ from $X$ is $X$-moderate, if there exist $N\in\mathbb{N}_{0}$ and $c>0$ such that
\begin{equation*}
\Vert f_{\varepsilon}\Vert_{X} \leq c\omega(\varepsilon)^{-N}.
\end{equation*}
\end{defn}
In what follows, we will use particular cases of $X$. Namely, ${H^1}$-moderate, ${L^{\infty}}$-moderate, and  $C(\left[0,T\right];H^{1})$-moderate families. For the last, we will shortly write $C$-moderate.

\begin{rem}
By assumptions, $(u_{0,\varepsilon})_{\varepsilon}$ and $(q_{\varepsilon})_{\varepsilon}$ are moderate.
\end{rem}

Now we will fix a notation. By writing $q\geq 0$, we mean that all regularisations $q_\varepsilon$ in our calculus are non-negative functions.
\begin{defn}
Let $q\geq 0$. The net $(u_{\varepsilon})_{\varepsilon}$ is said to be a very weak solution to the Cauchy problem (\ref{Equation}), if there exist an ${L^{\infty}}$-moderate regularisation $(q_{\varepsilon})_{\varepsilon}$ of the coefficient $q$ and $H^1$-moderate regularisation $(u_{0,\varepsilon})_{\varepsilon}$ of the initial function $u_0$, such that $(u_{\varepsilon})_{\varepsilon}$ solves the regularized equation
\begin{equation}
\label{Regularized equation}
\partial_{t}u_{\varepsilon}(t,x)-\Delta u_{\varepsilon}(t,x) + q_{\varepsilon}(x)u_{\varepsilon}(t,x)=0, ~~~(t,x)\in\left(0,T\right)\times \mathbb{R}^{d},
\end{equation}
with the Cauchy data $u_{\varepsilon}(0,x)=u_{0,\varepsilon}(x),$ for all $\varepsilon\in\left(0,1\right]$, and is $C$-moderate.
\end{defn}

With this setup the existence of a very weak solution becomes straightforward. But we will also analyse its properties later on.
\begin{thm}[Existence of a very weak solution]
Let $q\geq 0$. Assume that the regularisations of the coefficient $q$ and the Cauchy data $u_{0}$ satisfy the assumptions (\ref{Moderetness hyp data}) and (\ref{Moderetness hyp coeff}). Then the Cauchy problem (\ref{Equation}) has a very weak solution.
\end{thm}
\begin{proof}
Using the moderateness assumptions (\ref{Moderetness hyp data}), (\ref{Moderetness hyp coeff}), and the energy estimate (\ref{Energy estimate}), we arrive at
\begin{align*}
\Vert u_{\varepsilon}(t,\cdot)\Vert & \lesssim \omega(\varepsilon)^{-N} \times \omega(\varepsilon)^{-N_{0}}\\
& \lesssim \omega(\varepsilon)^{-N-N_{0}},
\end{align*}
concluding that $(u_{\varepsilon})_{\varepsilon}$ is $C$-moderate.
\end{proof}

\subsection{Uniqueness results}
In this subsection we discuss uniqueness of the very weak solution to the Cauchy problem (\ref{Equation}) for different cases of regularity of the potential $q$.

\subsubsection{\textbf{The classical case}} In the case when $q\in C^{\infty}(\mathbb{R}^{d})$, we require further conditions on the mollifiers, to ensure the uniqueness.

In the sequel, we are interested in the families of mollifiers with "$n$" vanishing moments. Let us define them as in the following.

\begin{defn} \label{defn moments}
\leavevmode
\begin{itemize}
    \item We denote by $\mathbb{A}_{n}$, the set of mollifiers defined by
\begin{equation}
\mathbb{A}_{n}=\left\lbrace \text{Friedrichs-mollifiers } \psi ~:~ \int_{\mathbb{R}^d}x^{k}\psi(x)dx=0 ~\text{ for }~ 1\leq k\leq n\right\rbrace. \label{Moments condition}
\end{equation}
    \item We say that $\psi\in \mathbb{A}_{\infty}$, if $\psi\in \mathbb{A}_{n}$ for all $n\in \mathbb{N}$.
\end{itemize}
\end{defn}

\begin{rem}
To construct such sets of mollifiers, we consider a Friedrichs-mollifier $\psi$ and set
\begin{equation*}
\Phi(x)=a_{0}\psi(x)+a_{1}\psi^{\prime}(x)+...+a_{n-1}\psi^{n-1}(x),
\end{equation*}
where the constants $a_{0}, \dots, a_{n-1}$ are determined by the conditions in (\ref{Moments condition}).
\end{rem}

\begin{lem} \label{lem q_eps-q}
For $N\in \mathbb{N}$, let $\psi\in\mathbb{A}_{N-1}$ and assume that $q\in C^{\infty}(\mathbb{R}^{d})$. Then, the estimate
\begin{equation}
\vert q_{\varepsilon}(x)-q(x)\vert \leq C\omega^{N}(\varepsilon) \label{Estimate q_eps-q}
\end{equation}
holds true for all $x\in \mathbb{R}^{d}$.
\end{lem}

\begin{proof}
Let $x\in \mathbb{R}^{d}$. We have
\begin{equation*}
\vert q_{\varepsilon}(x)-q(x)\vert \leq \omega^{-d}(\varepsilon)\int_{\mathbb{R}^{d}}\vert q(y)-q(x)\vert\psi\left(\omega^{-1}(\varepsilon)(y-x)\right) dy.
\end{equation*}
Making the change $z=\omega^{-d}(\varepsilon)(y-x)$, we get
\begin{equation*}
\vert q_{\varepsilon}(x)-q(x)\vert \leq \int_{\mathbb{R}^{d}}\vert q(x+\omega(\varepsilon)z)-q(x)\vert\psi(z) dz.
\end{equation*}
Expanding $q$ to order $N-1$, we get
\begin{equation*}
q(x+\omega(\varepsilon)z)-q(x)=\sum_{k=0}^{N} \frac{1}{(k-1)!}D^{(k-1)}q(x)(\omega(\varepsilon)z)^{k-1}+\mathcal{O}(\omega^{N}(\varepsilon)).
\end{equation*}
We get our estimate provided that the first $N-1$ moments of the mollifier $\psi$ vanish, finishing the proof of the lemma.
\end{proof}


To make things clear in what follows, we briefly repeat our regularisation nets. We regularise the coefficient $q$ and the initial data $u_0$ using suitable mollifiers $\psi, \Tilde{\psi}$, generating families of smooth functions $(q_{\varepsilon})_{\varepsilon}, (\Tilde{q}_{\varepsilon})_{\varepsilon}$ and $(u_{0,\varepsilon})_{\varepsilon}, (\Tilde{u}_{0,\varepsilon})_{\varepsilon}$. Namely,
\begin{equation*}
q_{\varepsilon}(x)=q\ast \psi_{\varepsilon }(x),~~~~u_{0,\varepsilon}(x)=u_{0}\ast \psi_{\varepsilon }(x),
\end{equation*}
\begin{equation*}
\Tilde{q}_{\varepsilon}(x)=q\ast \Tilde{\psi}_{\varepsilon }(x),~~~~\Tilde{u}_{0,\varepsilon}(x)=u_{0}\ast \Tilde{\psi}_{\varepsilon }(x),
\end{equation*}
where
\begin{equation*}
\psi_{\varepsilon }(x)=\omega(\varepsilon)^{-d}\psi(x/\omega(\varepsilon)), ~~~\varepsilon\in\left(0,1\right],
\end{equation*}
\begin{equation*}
\Tilde{\psi}_{\varepsilon }(x)=\omega(\varepsilon)^{-d}\Tilde{\psi}(x/\omega(\varepsilon)), ~~~\varepsilon\in\left(0,1\right],
\end{equation*}
and $\omega(\varepsilon)$ is a positive function converging to $0$ as $\varepsilon \rightarrow 0$ to be chosen later.

\begin{defn}
We say that the very weak solution to the Cauchy problem (\ref{Equation}) is unique, if for all $\psi, \Tilde{\psi}\in \mathbb{A}_{\infty}$, such that
\begin{equation}\label{A-negl}
    \Vert u_{0,\varepsilon} - \Tilde{u}_{0,\varepsilon}\Vert_{L^2} \lesssim \omega^{k}(\varepsilon) \,\, \left(\hbox{and} \,\, \Vert q_{\varepsilon} - \Tilde{q}_{\varepsilon}\Vert_{L^{\infty}} \lesssim \omega^{k}(\varepsilon)\right),
\end{equation}
for all $k>0$, we have
\begin{equation*}
\Vert u_{\varepsilon}(t,\cdot)-\Tilde{u}_{\varepsilon}(t,\cdot)\Vert_{L^{2}} \leq \omega^{N}(\varepsilon),
\end{equation*}
for all $N\in \mathbb{N}$, where $(u_{\varepsilon})_{\varepsilon}$ and $(\Tilde{u}_{\varepsilon})_{\varepsilon}$ solve, respectively, the families of the Cauchy problems
\begin{equation*}
\left\lbrace
\begin{array}{l}
\partial_{t}u_{\varepsilon}(t,x)-\Delta u_{\varepsilon}(t,x) + q_{\varepsilon}(x)u_{\varepsilon}(t,x)=0 ,~~~(t,x)\in\left[0,T\right]\times \mathbb{R}^{d},\\
u_{\varepsilon}(0,x)=u_{0,\varepsilon}(x),
\end{array}
\right.
\end{equation*}
and
\begin{equation*}
\left\lbrace
\begin{array}{l}
\partial_{t}\Tilde{u}_{\varepsilon}(t,x)-\Delta \Tilde{u}_{\varepsilon}(t,x) + \Tilde{q}_{\varepsilon}(x)\Tilde{u}_{\varepsilon}(t,x)=0 ,~~~(t,x)\in\left[0,T\right]\times \mathbb{R}^{d},\\
\Tilde{u}_{\varepsilon}(0,x)=\Tilde{u}_{0,\varepsilon}(x).
\end{array}
\right.
\end{equation*}

Also, the families of functions satisfying the properties \eqref{A-negl}, we call $\mathbb A_{\infty}$--negligible initial functions and coefficients, respectively.
\end{defn}

\begin{rem}
\label{Rem-negl}
We note that for any two $\psi, \Tilde{\psi}\in \mathbb{A}_{\infty}$ the difference of the corresponding regularisations of the coefficient $q\in C^{\infty}(\mathbb{R}^{d})$ is an $\mathbb A_{\infty}$--negligible function, that is,
$$
\Vert q_{\varepsilon} - \Tilde{q}_{\varepsilon}\Vert_{L^{\infty}}\lesssim \omega^{k}(\varepsilon),
$$
for all $k>0$, for all $\varepsilon\in(0, 1]$. Moreover, $(q_{\varepsilon}-q)_{\varepsilon\in(0, 1]}$ is also an $\mathbb A_{\infty}$--negligible family of functions.
\end{rem}

Note that the result of this remark holds for smooth functions. But in general, it also makes sense for other classes of regular functions and distributions. For more detailed analysis on the topic, the readers are referred to the paper \cite{GR15}.

\begin{thm}
\label{thm unicity classic}
Let $T>0$. Assume that a non-negative function $q\in C^{\infty}(\mathbb{R}^{d})$ and $u_{0}\in H^{1}(\mathbb{R}^{d})$ satisfy the conditions (\ref{Moderetness hyp data}) and (\ref{Moderetness hyp coeff}), respectively. Then, the very weak solution of the Cauchy problem (\ref{Equation}) is unique.
\end{thm}

\begin{proof}
Let $\psi, \Tilde{\psi}\in \mathbb{A}_{\infty}$ and consider $(q_{\varepsilon})_{\varepsilon}, (\Tilde{q}_{\varepsilon})_{\varepsilon}$ and $(u_{0,\varepsilon})_{\varepsilon}$, $(\Tilde{u}_{0,\varepsilon})_{\varepsilon}$ the regularisations of the coefficient $q$ and the data $u_0$ with respect to $\psi$ and $\Tilde{\psi}$. Assume that
\begin{equation}
    \Vert u_{0,\varepsilon} - \Tilde{u}_{0,\varepsilon}\Vert_{L^2} \leq C_{k}\omega^{k}(\varepsilon),
\end{equation}
for all $k>0$. Then, $u_{\varepsilon}$ and $\Tilde{u_{\varepsilon}}$, the solutions to the related Cauchy problems, satisfy the equation
\begin{equation}
\left\lbrace
\begin{array}{l}
\partial_{t}(u_{\varepsilon}-\Tilde{u}_{\varepsilon})(t,x)-\Delta (u_{\varepsilon}-\Tilde{u}_{\varepsilon})(t,x) + q_{\varepsilon}(x)(u_{\varepsilon}-\Tilde{u}_{\varepsilon})(t,x)=f_{\varepsilon}(t,x),\\
(u_{\varepsilon}-\Tilde{u}_{\varepsilon})(0,x)=(u_{0,\varepsilon}-\Tilde{u}_{0,\varepsilon})(x), \label{Equation uniqueness classic}
\end{array}
\right.
\end{equation}
with
\begin{equation*}
f_{\varepsilon}(t,x)=(\Tilde{q}_{\varepsilon}(x)-q_{\varepsilon}(x))\Tilde{u}_{\varepsilon}(t,x).
\end{equation*}
Let us denote by $U_{\varepsilon}(t,x):=u_{\varepsilon}(t,x)-\Tilde{u}_{\varepsilon}(t,x)$ the solution to the problem (\ref{Equation uniqueness classic}). Using Duhamel's principle, $U_{\varepsilon}$ is given by
\begin{equation*}
U_{\varepsilon}(t, x)=W_{\varepsilon}(t, x) + \int_{0}^{t}V_{\varepsilon}(x,t-s;s)ds,
\end{equation*}
where $W_{\varepsilon}(t, x)$ is the solution to the problem
\begin{equation*}
\left\lbrace
\begin{array}{l}
\partial_{t}W_{\varepsilon}(t, x)-\Delta W_{\varepsilon}(t, x) + q_{\varepsilon}(x)W_{\varepsilon}(t, x)=0,\\
W_{\varepsilon}(0, x)=(u_{0,\varepsilon}-\Tilde{u}_{0,\varepsilon})(x),
\end{array}
\right.
\end{equation*}
and $V_{\varepsilon}(x,t;s)$ solves
\begin{equation*}
\left\lbrace
\begin{array}{l}
\partial_{t}V_{\varepsilon}(x,t;s)-\Delta V_{\varepsilon}(x,t;s) + q_{\varepsilon}(x)V_{\varepsilon}(x,t;s)=0,\\
V_{\varepsilon}(x,0;s)=f_{\varepsilon}(s,x).
\end{array}
\right.
\end{equation*}
Taking $U_{\varepsilon}$ in $L^{2}$-norm and using (\ref{Energy estimate 2}) to estimate $V_{\varepsilon}$ and $W_{\varepsilon}$, we arrive at
\begin{align*}
\Vert U_{\varepsilon}(t, \cdot)\Vert_{L^2} & \leq \Vert W_{\varepsilon}(t, \cdot)\Vert_{L^2} + \int_{0}^{T}\Vert V_{\varepsilon}(\cdot,t-s;s)\Vert_{L^2} ds\\
& \lesssim \Vert u_{0,\varepsilon}-\Tilde{u}_{0,\varepsilon}\Vert_{L^2} + \int_{0}^{T}\Vert f_{\varepsilon}(s,\cdot)\Vert_{L^2} ds\\
& \lesssim \Vert u_{0,\varepsilon}-\Tilde{u}_{0,\varepsilon}\Vert_{L^2} + \Vert \Tilde{q}_{\varepsilon}-q_{\varepsilon}\Vert_{L^{\infty}}\int_{0}^{T}\Vert \Tilde{u}_{\varepsilon}(s,\cdot)\Vert_{L^2} ds.
\end{align*}
The net $(\tilde{u}_{\varepsilon})_{\varepsilon}$ is moderate, the uniqueness of the very weak solution follows by the assumption that $(u_{0,\varepsilon} - \Tilde{u}_{0,\varepsilon})_{\varepsilon\in(0, 1]}$ is an $\mathbb A_{\infty}$--negligible family of initial functions, that is,
\begin{equation*}
\Vert u_{0,\varepsilon} - \Tilde{u}_{0,\varepsilon}\Vert_{L^2} \leq C_{k}\omega^{k}(\varepsilon) \text{~~~for all~~} k>0,
\end{equation*}
the application of Lemma \ref{lem q_eps-q} and Remark \eqref{Rem-negl} due to the $\mathbb A_{\infty}$--negligibly of the family of coefficients
$\Tilde{q}_{\varepsilon}$ and $q_{\varepsilon}$. This ends the proof of the theorem.
\end{proof}

\subsubsection{\textbf{The singular case}}
In the case when $q$ is singular, we prove uniqueness in the sense of the following definition.
\begin{defn} \label{defn:uniqueness singular case}
We say that the very weak solution to the Cauchy problem (\ref{Equation}) is unique, if for all families $(q_{\varepsilon})_{\varepsilon}$, $(\Tilde{q}_{\varepsilon})_{\varepsilon}$ and $(u_{0,\varepsilon})_{\varepsilon}$, $(\Tilde{u}_{0,\varepsilon})_{\varepsilon}$, regularisations of the coefficient $q$ and $u_0$, satisfying
\begin{equation*}
\Vert q_{\varepsilon}-\Tilde{q}_{\varepsilon}\Vert_{L^{\infty}}\leq C_{k}\varepsilon^{k} \text{~~for all~~} k>0
\end{equation*}
and
\begin{equation*}
\Vert u_{0,\varepsilon}-\Tilde{u}_{0,\varepsilon}\Vert_{L^{2}}\leq C_{l}\varepsilon^{l} \text{~~for all~~} l>0,
\end{equation*}
then
\begin{equation*}
\Vert u_{\varepsilon}(t,\cdot)-\Tilde{u}_{\varepsilon}(t,\cdot)\Vert_{L^{2}} \leq C_{N}\varepsilon^{N},
\end{equation*}
for all $N>0$, where $(u_{\varepsilon})_{\varepsilon}$ and $(\Tilde{u}_{\varepsilon})_{\varepsilon}$ solve, respectively, the families of the Cauchy problems
\begin{equation*}
\left\lbrace
\begin{array}{l}
\partial_{t}u_{\varepsilon}(t,x)-\Delta u_{\varepsilon}(t,x) + q_{\varepsilon}(x)u_{\varepsilon}(t,x)=0 ,~~~(t,x)\in\left[0,T\right]\times \mathbb{R}^{d},\\
u_{\varepsilon}(0,x)=u_{0,\varepsilon}(x),
\end{array}
\right.
\end{equation*}
and
\begin{equation*}
\left\lbrace
\begin{array}{l}
\partial_{t}\Tilde{u}_{\varepsilon}(t,x)-\Delta \Tilde{u}_{\varepsilon}(t,x) + \Tilde{q}_{\varepsilon}(x)\Tilde{u}_{\varepsilon}(t,x)=0 ,~~~(t,x)\in\left[0,T\right]\times \mathbb{R}^{d},\\
\Tilde{u}_{\varepsilon}(0,x)=\Tilde{u}_{0,\varepsilon}(x).
\end{array}
\right.
\end{equation*}
\end{defn}

We note that in particular the hypotheses of this definition are fulfilled when $q$ is smooth. But it is not the only case. For more suitable examples of the coefficient $q$, we refer to \cite{GR15}, where a number of classes of regular and distributional $q$ are analysed.

\begin{thm}\label{thm uniqueness}
Let $T>0$. Assume that $q\geq0$ and $u_{0}\in H^{1}(\mathbb{R}^{d})$ satisfy the moderateness assumptions (\ref{Moderetness hyp data}) and (\ref{Moderetness hyp coeff}), respectively. Then, the very weak solution to the Cauchy problem (\ref{Equation}) is unique.
\end{thm}

\begin{proof}
Let $(q_{\varepsilon})_{\varepsilon}$, $(\Tilde{q}_{\varepsilon})_{\varepsilon}$ and $(u_{0,\varepsilon})_{\varepsilon}$, $(\Tilde{u}_{0,\varepsilon})_{\varepsilon}$, regularisations of the coefficient $q$ and the data $u_0$, satisfying
\begin{equation*}
\Vert q_{\varepsilon}-\Tilde{q}_{\varepsilon}\Vert_{L^{\infty}}\leq C_{k}\varepsilon^{k}, \text{~~for all~~} k>0,
\end{equation*}
and
\begin{equation*}
\Vert u_{0,\varepsilon}-\Tilde{u}_{0,\varepsilon}\Vert_{L^{2}}\leq C_{l}\varepsilon^{l}, \text{~~for all~~} l>0.
\end{equation*}
Then, $(u_{\varepsilon})_{\varepsilon}$ and $(\Tilde{u}_{\varepsilon})_{\varepsilon}$, the solutions to the related Cauchy problems, satisfy
\begin{equation}
\left\lbrace
\begin{array}{l}
\partial_{t}(u_{\varepsilon}-\Tilde{u}_{\varepsilon})(t,x)-\Delta (u_{\varepsilon}-\Tilde{u}_{\varepsilon})(t,x) + q_{\varepsilon}(x)(u_{\varepsilon}-\Tilde{u}_{\varepsilon})(t,x)=f_{\varepsilon}(t,x),\\
(u_{\varepsilon}-\Tilde{u}_{\varepsilon})(0,x)=(u_{0,\varepsilon}-\Tilde{u}_{0,\varepsilon})(x), \label{Equation uniqueness}
\end{array}
\right.
\end{equation}
with
\begin{equation*}
f_{\varepsilon}(t,x)=(\Tilde{q}_{\varepsilon}(x)-q_{\varepsilon}(x))\Tilde{u}_{\varepsilon}(t,x).
\end{equation*}
Let us denote by $U_{\varepsilon}(t,x):=u_{\varepsilon}(t,x)-\Tilde{u}_{\varepsilon}(t,x)$ the solution to the equation (\ref{Equation uniqueness}). Using similar arguments as in Theorem \ref{thm unicity classic}, we get
\begin{equation*}
\Vert U_{\varepsilon}(t, \cdot)\Vert_{L^2} \lesssim \Vert u_{0,\varepsilon}-\Tilde{u}_{0,\varepsilon}\Vert_{L^2} + \Vert \Tilde{q}_{\varepsilon}-q_{\varepsilon}\Vert_{L^{\infty}}\int_{0}^{T}\Vert \Tilde{u}_{\varepsilon}(s,\cdot)\Vert_{L^2}.
\end{equation*}
The family $(\Tilde{u}_{\varepsilon})_{\varepsilon}$ is a very weak solution to the Cauchy problem (\ref{Equation}), it is then moderate, i.e. there exists $N_0 \in \mathbb{N}_{0}$ such that
\begin{equation*}
\Vert \Tilde{u}_{\varepsilon}(s,\cdot)\Vert_{L^2} \leq c\omega^{-N_0}(\varepsilon).
\end{equation*}
On the other hand, we have that $\Vert q_{\varepsilon}-\Tilde{q}_{\varepsilon}\Vert_{L^{\infty}}\leq C_{k}\varepsilon^{k}$, for all $k>0$, and $\Vert u_{0,\varepsilon}-\Tilde{u}_{0,\varepsilon}\Vert_{L^{2}}\leq C_{l}\varepsilon^{l}$, for all $l>0$. Thus, we obtain that
\begin{equation*}
\Vert U_{\varepsilon}(t, \cdot)\Vert_{L^2}:=\Vert u_{\varepsilon}(t,\cdot)-\Tilde{u}_{\varepsilon}(t,\cdot)\Vert_{L^2} \lesssim \varepsilon^{N},
\end{equation*}
for all $N>0$, showing the uniqueness of the very weak solution.
\end{proof}

\subsection{Consistency with the classical case}
Now we show that if the classical solution of the Cauchy problem (\ref{Equation}) given by Lemma \ref{Lemma 1} exists then the very weak solution recaptures it.

\begin{thm} \label{thm:consistency positive case}
Let $u_{0}\in H^{1}(\mathbb{R}^{d})$. Assume that $q\in L^{\infty}(\mathbb{R}^{d})$ is non-negative and consider the Cauchy problem
\begin{equation}
\left\lbrace
\begin{array}{l}
u_{t}(t,x)-\Delta u(t,x) + q(x)u(t,x)=0 ,~~~(t,x)\in\left[0,T\right]\times \mathbb{R}^{d},\\
u(0,x)=u_{0}(x). \label{Equation with reg. coeff}
\end{array}
\right.
\end{equation}
Let $(u_{\varepsilon})_{\varepsilon}$ be a very weak solution of (\ref{Equation with reg. coeff}). Then, for any regularising families $(q_{\varepsilon})_{\varepsilon}$ and $(u_{0,\varepsilon})_{\varepsilon}$, the net $(u_{\varepsilon})_{\varepsilon}$ converges in $L^{2}$ as $\varepsilon \rightarrow 0$ to the classical solution of the Cauchy problem (\ref{Equation with reg. coeff}).
\end{thm}

\begin{proof}
Consider the classical solution $u$ to
\begin{equation*}
\left\lbrace
\begin{array}{l}
u_{t}(t,x)-\Delta u(t,x) + q(x)u(t,x)=0 ,~~~(t,x)\in\left[0,T\right]\times \mathbb{R}^{d},\\
u(0,x)=u_{0}(x).
\end{array}
\right.
\end{equation*}
Note that for the very weak solution there is a representation $(u_{\varepsilon})_{\varepsilon}$ such that
\begin{equation*}
\left\lbrace
\begin{array}{l}
\partial_{t}u_{\varepsilon}(t,x)-\Delta u_{\varepsilon}(t,x) + q_{\varepsilon}(x)u_{\varepsilon}(t,x)=0 ,~~~(t,x)\in\left[0,T\right]\times \mathbb{R}^{d},\\
u_{\varepsilon}(0,x)=u_{0,\varepsilon}(x).
\end{array}
\right.
\end{equation*}
Taking the difference, we get
\begin{equation}
\left\lbrace
\begin{array}{l}
\partial_{t}(u-u_{\varepsilon})(t,x)-\Delta (u-u_{\varepsilon})(t,x) + q_{\varepsilon}(x)(u-u_{\varepsilon})(t,x)=\eta_{\varepsilon}(t,x),\\
(u-u_{\varepsilon})(0,x)=(u_{0}-u_{0,\varepsilon})(x), \label{Equation consistency}
\end{array}
\right.
\end{equation}
where
\begin{equation*}
\eta_{\varepsilon}(t,x)=(q_{\varepsilon}(x)-q(x))u(t,x).
\end{equation*}
Let us denote $U_{\varepsilon}(t,x):=(u-u_{\varepsilon})(t,x)$ and let $W_{\varepsilon}(t, x)$ be the solution to the auxiliary homogeneous problem
\begin{equation*}
\left\lbrace
\begin{array}{l}
\partial_{t}W_{\varepsilon}(t, x)-\Delta W_{\varepsilon}(t, x) + q_{\varepsilon}(x)W_{\varepsilon}(t, x)=0,\\
W_{\varepsilon}(0, x)=(u_{0}-u_{0,\varepsilon})(x).
\end{array}
\right.
\end{equation*}
Then, by Duhamel's principle, the solution to (\ref{Equation consistency}) is given by
\begin{equation}
U_{\varepsilon}(t, x)=W_{\varepsilon}(t, x) + \int_{0}^{t}V_{\varepsilon}(x,t-s;s)ds, \label{Duhamel consistency}
\end{equation}
where $V_{\varepsilon}(x,t;s)$ is the solution to the problem
\begin{equation*}
\left\lbrace
\begin{array}{l}
\partial_{t}V_{\varepsilon}(x,t;s)-\Delta V_{\varepsilon}(x,t;s) + q_{\varepsilon}(x)V_{\varepsilon}(x,t;s)=0,\\
V_{\varepsilon}(x,0;s)=\eta_{\varepsilon}(t,x).
\end{array}
\right.
\end{equation*}
As in Theorem \ref{thm uniqueness}, taking the $L^{2}$-norm in (\ref{Duhamel consistency}) and using (\ref{Energy estimate 2}) to estimate $V_{\varepsilon}$ and $W_{\varepsilon}$, we get
\begin{align*}
\Vert U_{\varepsilon}(t, \cdot)\Vert_{L^2} & \leq \Vert W_{\varepsilon}(t, \cdot)\Vert_{L^2} + \int_{0}^{T}\Vert V_{\varepsilon}(\cdot,t-s;s)\Vert_{L^2} ds\\
& \lesssim \Vert u_{0}-u_{0,\varepsilon}\Vert_{L^2} + \int_{0}^{T}\Vert \eta_{\varepsilon}(s,\cdot)\Vert_{L^2} ds\\
& \lesssim \Vert u_{0}-u_{0,\varepsilon}\Vert_{L^2} + \Vert q_{\varepsilon}-q\Vert_{L^{\infty}}\int_{0}^{T}\Vert u(s,\cdot)\Vert_{L^2} ds,
\end{align*}
and taking into account that
\begin{equation*}
\Vert q_{\varepsilon}-q\Vert_{L^{\infty}} \rightarrow 0 \text{~~as~~} \varepsilon\rightarrow 0
\end{equation*}
and
\begin{equation*}
\Vert u_{0,\varepsilon}-u_{0}\Vert_{L^{2}} \rightarrow 0 \text{~~as~~} \varepsilon\rightarrow 0,
\end{equation*}
consequently, it implies that $u_{\varepsilon}$ converges to $u$ in $L^{2}$ as $\varepsilon\to0$.
\end{proof}

\section{Part II: Negative potential}
\label{NP}
In this part we aim to study the case when the potential is negative and to show that the problem is still well-posed. Namely, we consider the Cauchy problem for the heat equation
\begin{equation}
\label{Equation 2}
\left\lbrace
\begin{array}{l}
\partial_{t}u(t,x)-\Delta u(t,x) - q(x)u(t,x)=0, \,\,\,(t,x)\in\left(0,T\right)\times \mathbb{R}^{d}, \\
u(0,x)=u_{0}(x),
\end{array}
\right.
\end{equation}
where $q$ is non-negative.

In the classical case, we have the following energy estimates for the solution of the problem \eqref{Equation 2}.
\begin{lem} \label{Lemma 3}
Let $u_{0}\in L^{2}(\mathbb{R}^d)$ and suppose that $q\in L^{\infty}(\mathbb{R}^d)$ is non-negative. Then, there is a unique solution $u\in C(\left[0,T\right];L^{2})$ to (\ref{Equation 2}) and it satisfies the estimate
\begin{equation}
    \Vert u(t,\cdot)\Vert_{L^2} \lesssim \exp{\left( t\Vert q\Vert_{L^{\infty}} \right)} \Vert u_0\Vert_{L^2}, \label{Energy estimate 3}
\end{equation}
for all $t\in [0,T]$.
\end{lem}

\begin{proof}
Multiplying the equation in (\ref{Equation 2}) by $u$, integrating with respect to $x$, and taking the real part, we obtain
\begin{equation*}
    Re \left(\langle u_{t}(t,\cdot),u(t,\cdot)\rangle_{L^2} + \langle -\Delta u(t,\cdot),u(t,\cdot)\rangle_{L^2} - \langle q(\cdot)u(t,\cdot),u(t,\cdot)\rangle_{L^2} \right)=0,
\end{equation*}
for all $t\in [0,T]$. Using similar arguments as in Lemma \ref{Lemma 1} and noting that the term $\Vert q(\cdot)u(t,\cdot)\Vert_{L^2}$ can be estimated by $\Vert q\Vert_{L^{\infty}} \Vert u(t,\cdot)\Vert_{L^2}$, we get
\begin{equation*}
    \partial_{t}\Vert u(t,\cdot)\Vert_{L^2} \lesssim \Vert q\Vert_{L^{\infty}} \Vert u(t,\cdot)\Vert_{L^2},
\end{equation*}
for all $t\in [0,T]$. The desired estimate follows by the application of Gronwall's lemma.
\end{proof}

Let now assume that the potential $q$ and the initial data $u_0$ are singular. Consider the Cauchy problem for the heat equation
\begin{equation}
\label{Equation 3}
\left\lbrace
\begin{array}{l}
\partial_{t}u(t,x)-\Delta u(t,x) - q(x)u(t,x)=0, \,\,\,(t,x)\in\left(0,T\right)\times \mathbb{R}^{d}, \\
u(0,x)=u_{0}(x).
\end{array}
\right.
\end{equation}
In order to prove the existence of a very weak solution to (\ref{Equation 3}), we proceed as in the case of the positive potential. We start by regularising the equation in (\ref{Equation 3}). In other words, using
\begin{equation*}
    \psi_{\varepsilon }(x)=\omega(\varepsilon)^{-d}\psi(x/\omega(\varepsilon)), ~~~\varepsilon\in\left(0,1\right],
\end{equation*}
where $\psi$ is a Friedrichs mollifier and $\omega$ is a positive function converging to $0$ as $\varepsilon \rightarrow 0$, to be chosen later, we regularise $q$ and $u_0$ obtaining the nets $(q_{\varepsilon})_{\varepsilon}=(q\ast\psi_{\varepsilon})_{\varepsilon}$ and $(u_{0,\varepsilon})_{\varepsilon}=(u_0\ast\psi_{\varepsilon})_{\varepsilon}$. For this, we can assume that $q$ and $u_0$ are distributions.

\begin{assum}
\label{Assump_neg}
We assume that there exist $N_0, N_1 \in \mathbb{N}_0$ such that
\begin{equation}
\Vert q_{\varepsilon}\Vert_{L^{\infty}}\leq C_0\omega(\varepsilon)^{-N_0},
\label{Moderetness hyp coeff 1}
\end{equation}
and
\begin{equation}
\Vert u_{0,\varepsilon}\Vert_{L^2}\leq C_1\omega(\varepsilon)^{-N_1}. \label{Moderetness hyp data 1}
\end{equation}
\end{assum}

\subsection{Existence of very weak solutions}
In this subsection we give the definition of a very weak solution adapted to the problem (\ref{Equation 3}). For this, we will make use of the same definition of the moderateness as in the non-negative case. Nevertheless, let us recall it here.
\begin{defn}[Moderateness]
\label{Def:Moderetness 1}
Let $X$ be a Banach space with the norm $\|\cdot\|_{X}$. Then we say that a net of functions $(f_{\varepsilon})_{\varepsilon}$ from $X$ is $X$-moderate, if there exist $N\in\mathbb{N}_{0}$ and $c>0$ such that
\begin{equation*}
\Vert f_{\varepsilon}\Vert_{X} \leq c\omega(\varepsilon)^{-N}.
\end{equation*}
\end{defn}
In what follows, we will use particular cases of $X$. Namely, ${L^2}$-moderate, ${L^{\infty}}$-moderate, and  $C(\left[0,T\right];L^{2})$-moderate families. For the last, we will shortly write $C$-moderate.

\begin{defn}
Let $q$ be non-negative. Then the net $(u_{\varepsilon})_{\varepsilon}$ is said to be a very weak solution to the problem (\ref{Equation 3}), if there exist an ${L^{\infty}}$-moderate regularisation $(q_{\varepsilon})_{\varepsilon}$ of the coefficient $q$ and an $L^2$-moderate regularisation $(u_{0,\varepsilon})_{\varepsilon}$ of $u_0$ such that $(u_{\varepsilon})_{\varepsilon}$ solves the regularized problem
\begin{equation}
\label{Regularized equation 1}
\left\lbrace
\begin{array}{l}
\partial_{t}u_{\varepsilon}(t,x)-\Delta u_{\varepsilon}(t,x) - q_{\varepsilon}(x)u_{\varepsilon}(t,x)=0, ~~~(t,x)\in\left(0,T\right)\times \mathbb{R}^{d},\\
u_{\varepsilon}(0,x)=u_{0,\varepsilon}(x),
\end{array}
\right.
\end{equation}
for all $\varepsilon\in\left(0,1\right]$, and is $C$-moderate.
\end{defn}

\begin{thm}[Existence of a very weak solution]
Let $q\geq 0$. Assume that the nets $(q_{\varepsilon})_{\varepsilon}$ and $(u_{0,\varepsilon})_{\varepsilon}$ satisfy the assumptions (\ref{Moderetness hyp coeff 1}) and (\ref{Moderetness hyp data 1}), respectively. Then the problem (\ref{Equation 3}) has a very weak solution.
\end{thm}

\begin{proof}
The nets $(q_{\varepsilon})_{\varepsilon}$ and $(u_{0,\varepsilon})_{\varepsilon}$ are moderate by the assumption. To prove that a very weak solution to the Cauchy problem (\ref{Equation 3}) exists, we need to show that the net $(u_{\varepsilon})_{\varepsilon}$, a solution to the regularized problem (\ref{Regularized equation 1}), is $C$-moderate. Indeed, using the assumptions (\ref{Moderetness hyp coeff 1}), (\ref{Moderetness hyp data 1}) and the estimate (\ref{Energy estimate 3}), we get
\begin{equation*}
    \Vert u(t,\cdot)\Vert_{L^2} \lesssim \exp{\left( t\omega(\varepsilon)^{-N_0}\right)}\omega(\varepsilon)^{-N_1},
\end{equation*}
for all $t\in [0,T]$. Choosing $\omega(\varepsilon)=\left( \log \varepsilon^{-N_0}\right)^{-\frac{1}{N_0}}$,
we obtain that
\begin{align*}
    \Vert u(t,\cdot)\Vert_{L^2} & \lesssim \varepsilon^{-tN_0}\times \left( \log \varepsilon^{-N_0} \right)^{\frac{N_1}{N_0}}\\
    & \lesssim \varepsilon^{-TN_0}\times \varepsilon^{-N_1},
\end{align*}
where the fact that $t\in [0,T]$ and that $\log \varepsilon^{-N_0}$ can be estimated by $\varepsilon^{-N_0}$ are used. Then the net $(u_{\varepsilon})_{\varepsilon}$ is $C$-moderate, implying the existence of very weak solutions.
\end{proof}

\subsection{Uniqueness results}

Here, we prove the uniqueness of the very weak solution to the heat equation with a non-positive potential \eqref{Equation 3} in the spirit of Definition \ref{defn:uniqueness singular case}, adapted to our problem.
\begin{defn} \label{defn:uniqueness_2}
Let the regularisations $(q_{\varepsilon})_{\varepsilon}$ and $(\Tilde{q}_{\varepsilon})_{\varepsilon}$ of $q$ and the regularisations $(u_{0,\varepsilon})_{\varepsilon}$ and $(\Tilde{u}_{0,\varepsilon})_{\varepsilon}$ of $u_0$ satisfy Assumption \ref{Assump_neg}. Then we say that the very weak solution to the heat equation (\ref{Equation 3}) is unique, if for all families $(q_{\varepsilon})_{\varepsilon}$, $(\Tilde{q}_{\varepsilon})_{\varepsilon}$ and $(u_{0,\varepsilon})_{\varepsilon}$, $(\Tilde{u}_{0,\varepsilon})_{\varepsilon}$, satisfying
\begin{equation*}
\Vert q_{\varepsilon}-\Tilde{q}_{\varepsilon}\Vert_{L^{\infty}}\leq C_{k}\varepsilon^{k} \text{~~for all~~} k>0
\end{equation*}
and
\begin{equation*}
\Vert u_{0,\varepsilon}-\Tilde{u}_{0,\varepsilon}\Vert_{L^{2}}\leq C_{l}\varepsilon^{l} \text{~~for all~~} l>0,
\end{equation*}
we have
\begin{equation*}
\Vert u_{\varepsilon}(t,\cdot)-\Tilde{u}_{\varepsilon}(t,\cdot)\Vert_{L^{2}} \leq C_{N}\varepsilon^{N}
\end{equation*}
for all $N>0$, where $(u_{\varepsilon})_{\varepsilon}$ and $(\Tilde{u}_{\varepsilon})_{\varepsilon}$ solve, respectively, the families of the Cauchy problems
\begin{equation*}
\left\lbrace
\begin{array}{l}
\partial_{t}u_{\varepsilon}(t,x)-\Delta u_{\varepsilon}(t,x) - q_{\varepsilon}(x)u_{\varepsilon}(t,x)=0 ,~~~(t,x)\in\left(0,T\right)\times \mathbb{R}^{d},\\
u_{\varepsilon}(0,x)=u_{0,\varepsilon}(x),
\end{array}
\right.
\end{equation*}
and
\begin{equation*}
\left\lbrace
\begin{array}{l}
\partial_{t}\Tilde{u}_{\varepsilon}(t,x)-\Delta \Tilde{u}_{\varepsilon}(t,x) - \Tilde{q}_{\varepsilon}(x)\Tilde{u}_{\varepsilon}(t,x)=0 ,~~~(t,x)\in\left(0,T\right)\times \mathbb{R}^{d},\\
\Tilde{u}_{\varepsilon}(0,x)=\Tilde{u}_{0,\varepsilon}(x).
\end{array}
\right.
\end{equation*}
\end{defn}

\begin{thm}\label{thm uniqueness negative case}
Let $T >0$. Assume that the nets $(q_{\varepsilon})_{\varepsilon}$ and $(u_{0,\varepsilon})_{\varepsilon}$ satisfy the assumptions (\ref{Moderetness hyp coeff 1}) and (\ref{Moderetness hyp data 1}), respectively. Then, the very weak solution to the Cauchy problem (\ref{Equation 3}) is unique.
\end{thm}

\begin{proof}
Let us consider $(q_{\varepsilon})_{\varepsilon}$, $(\Tilde{q}_{\varepsilon})_{\varepsilon}$ and $(u_{0,\varepsilon})_{\varepsilon}$, $(\Tilde{u}_{0,\varepsilon})_{\varepsilon}$, regularisations of the $q$ and $u_0$, satisfying
\begin{equation*}
\Vert q_{\varepsilon}-\Tilde{q}_{\varepsilon}\Vert_{L^{\infty}}\leq C_{k}\varepsilon^{k} 
\text{~~for all~~} k>0
\end{equation*}
and
\begin{equation*}
\Vert u_{0,\varepsilon}-\Tilde{u}_{0,\varepsilon}\Vert_{L^{2}}\leq C_{l}\varepsilon^{l} 
\text{~~for all~~} l>0.
\end{equation*}
Then, $(u_{\varepsilon})_{\varepsilon}$ and $(\Tilde{u}_{\varepsilon})_{\varepsilon}$, the solutions to the related Cauchy problems, satisfy
\begin{equation}
\left\lbrace
\begin{array}{l}
\partial_{t}(u_{\varepsilon}-\Tilde{u}_{\varepsilon})(t,x)-\Delta (u_{\varepsilon}-\Tilde{u}_{\varepsilon})(t,x) - q_{\varepsilon}(x)(u_{\varepsilon}-\Tilde{u}_{\varepsilon})(t,x)=f_{\varepsilon}(t,x),\\
(u_{\varepsilon}-\Tilde{u}_{\varepsilon})(0,x)=(u_{0,\varepsilon}-\Tilde{u}_{0,\varepsilon})(x), \label{Equation uniqueness 1}
\end{array}
\right.
\end{equation}
with
\begin{equation*}
f_{\varepsilon}(t,x)=(q_{\varepsilon}(x)-\Tilde{q}_{\varepsilon}(x))\Tilde{u}_{\varepsilon}(t,x).
\end{equation*}
Let us denote by $U_{\varepsilon}(t,x):=u_{\varepsilon}(t,x)-\Tilde{u}_{\varepsilon}(t,x)$ the solution to the equation (\ref{Equation uniqueness 1}). Arguing as in Theorem \ref{thm unicity classic} and using the estimate (\ref{Energy estimate 3}), we arrive at
\begin{equation*}
\Vert U_{\varepsilon}(t, \cdot)\Vert_{L^2} \lesssim \exp{\left( t\Vert q_{\varepsilon}\Vert_{L^{\infty}} \right)}\Vert u_{0,\varepsilon}-\Tilde{u}_{0,\varepsilon} \Vert_{L^2} + \Vert q_{\varepsilon}-\Tilde{q}_{\varepsilon} \Vert_{L^{\infty}}\int_{0}^{T}\exp{\left( s\Vert q_{\varepsilon}\Vert_{L^{\infty}} \right)}\Vert\Tilde{u}_{\varepsilon}(s,\cdot)\Vert_{L^2} ds.
\end{equation*}
On the one hand, the net $(q_{\varepsilon})_{\varepsilon}$ is moderate by the assumption and $(\Tilde{u}_{\varepsilon})_{\varepsilon}$ is moderate as a very weak solution. From the other hand, we have that
\begin{equation*}
\Vert q_{\varepsilon}-\Tilde{q}_{\varepsilon}\Vert_{L^{\infty}}\leq C_{k}\varepsilon^{k} 
\text{~~for all~~} k>0,
\end{equation*}
and
\begin{equation*}
\Vert u_{0,\varepsilon}-\Tilde{u}_{0,\varepsilon}\Vert_{L^{2}}\leq C_{l}\varepsilon^{l} 
\text{~~for all~~} l>0.
\end{equation*}
By choosing $\omega(\varepsilon)=\left( \log \varepsilon^{-N_0}\right)^{-\frac{1}{N_0}}$ for $q_{\varepsilon}$ in \eqref{Moderetness hyp coeff 1}, it follows that
\begin{equation*}
\Vert U_{\varepsilon}(t, \cdot)\Vert_{L^2}=\Vert u_{\varepsilon}(t,\cdot)-\Tilde{u}_{\varepsilon}(t,\cdot)\Vert_{L^2} \lesssim \varepsilon^{N}, 
\end{equation*}
for all $N>0$, ending the proof.
\end{proof}

\subsection{Consistency with the classical case}
We conclude this section by showing that if the coefficient and the Cauchy data are regular then the very weak solution coincides with the classical one, given by Lemma \ref{Lemma 3}.

\begin{thm}
Let $u_{0}\in L^{2}(\mathbb{R}^{d})$. Assume that $q\in L^{\infty}(\mathbb{R}^{d})$ is non-negative and consider the Cauchy problem
for the heat equation
\begin{equation}
\left\lbrace
\begin{array}{l}
u_{t}(t,x)-\Delta u(t,x) - q(x)u(t,x)=0 ,~~~(t,x)\in\left(0,T\right)\times \mathbb{R}^{d},\\
u(0,x)=u_{0}(x). \label{Equation with reg. coeff 1}
\end{array}
\right.
\end{equation}
Let $(u_{\varepsilon})_{\varepsilon}$ be a very weak solution of the heat equation (\ref{Equation with reg. coeff 1}). Then, for any regularising families $(q_{\varepsilon})_{\varepsilon}$ and $(u_{0,\varepsilon})_{\varepsilon}$, the net $(u_{\varepsilon})_{\varepsilon}$ converges in $L^{2}$ as $\varepsilon \rightarrow 0$ to the classical solution of the Cauchy problem (\ref{Equation with reg. coeff 1}).
\end{thm}

\begin{proof}
Let us denote the classical solution and the very weak one by $u$ and $(u_{\varepsilon})_{\varepsilon}$, respectively. It is clear, that they satisfy
\begin{equation*}
\left\lbrace
\begin{array}{l}
u_{t}(t,x)-\Delta u(t,x) - q(x)u(t,x)=0, \,\,\,(t,x)\in\left(0,T\right)\times \mathbb{R}^{d},\\
u(0,x)=u_{0}(x),
\end{array}
\right.
\end{equation*}
and
\begin{equation*}
\left\lbrace
\begin{array}{l}
\partial_{t}u_{\varepsilon}(t,x)-\Delta u_{\varepsilon}(t,x) - q_{\varepsilon}(x)u_{\varepsilon}(t,x)=0, \,\,\,(t,x)\in\left(0,T\right)\times \mathbb{R}^{d},\\
u_{\varepsilon}(0,x)=u_{0,\varepsilon}(x),
\end{array}
\right.
\end{equation*}
respectively. Let us denote by $V_{\varepsilon}(t,x):=(u_{\varepsilon}-u)(t,x)$. Using the estimate (\ref{Energy estimate 3}) and the same arguments as in the positive potential case, we show that
\begin{equation*}
\Vert V_{\varepsilon}(t, \cdot)\Vert_{L^2} \lesssim \exp{\left( t\Vert q_{\varepsilon}\Vert_{L^{\infty}} \right)}\Vert u_{0,\varepsilon}-u_{0} \Vert_{L^2} + \Vert q_{\varepsilon}-q \Vert_{L^{\infty}}\int_{0}^{T}\exp{\left( s\Vert q_{\varepsilon}\Vert_{L^{\infty}} \right)}\Vert u(s,\cdot)\Vert_{L^2} ds.
\end{equation*}
By taking into account that
\begin{equation*}
\Vert q_{\varepsilon}-q\Vert_{L^{\infty}} \rightarrow 0 \text{~~as~~} \varepsilon\rightarrow 0
\end{equation*}
and
\begin{equation*}
\Vert u_{0,\varepsilon}-u_{0}\Vert_{L^{2}} \rightarrow 0 \text{~~as~~} \varepsilon\rightarrow 0,
\end{equation*}
from the other hand, due to the facts $q_{\varepsilon}$ is bounded as a regularisation of an essentially bounded function and $\Vert u(s,\cdot)\Vert_{L^2}$ is bounded as well as $u$ is a classical solution, we conclude that $(u_{\varepsilon})_{\varepsilon}$ converges to $u$ in $L^{2}$ as $\varepsilon\to0$.
\end{proof}

\begin{figure}[ht!]
\begin{minipage}[h]{0.47\linewidth}
\center{\includegraphics[scale=0.35]{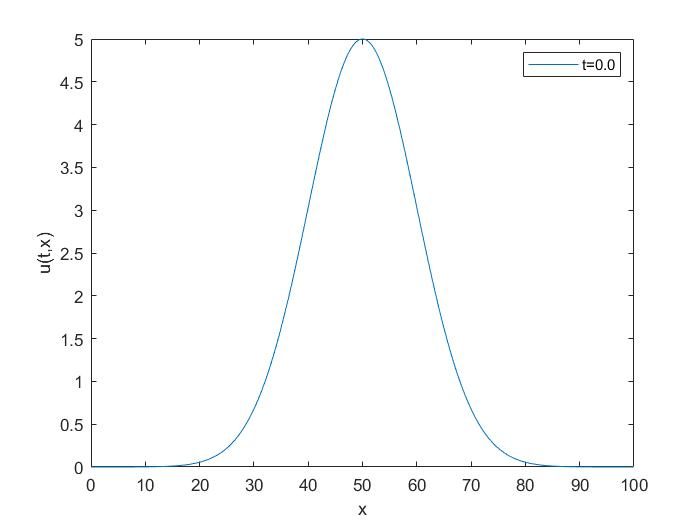}}
\end{minipage}
\hfill
\begin{minipage}[h]{0.47\linewidth}
\center{\includegraphics[scale=0.35]{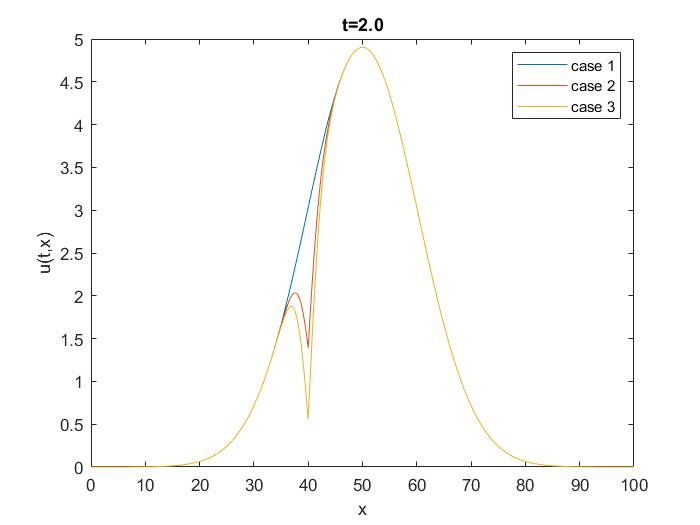}}
\end{minipage}
\hfill
\begin{minipage}[h]{0.47\linewidth}
\center{\includegraphics[scale=0.35]{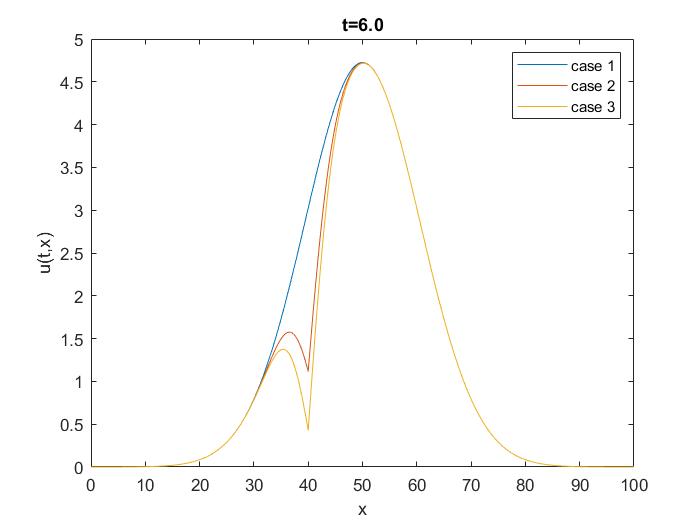}}
\end{minipage}
\hfill
\begin{minipage}[h]{0.47\linewidth}
\center{\includegraphics[scale=0.35]{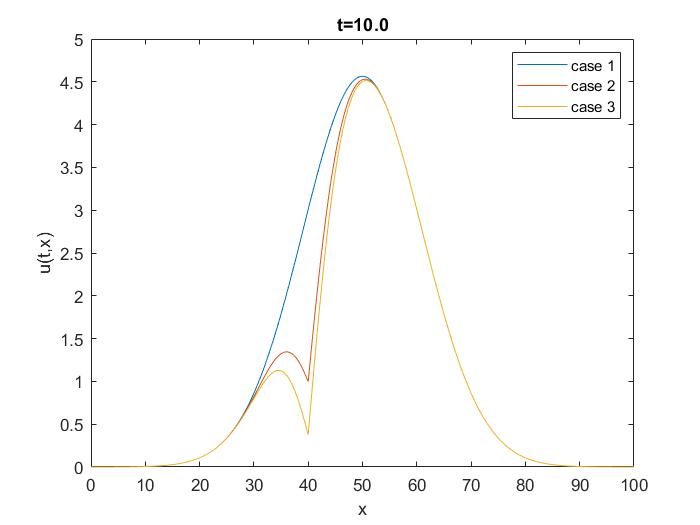}}
\end{minipage}
\caption{In these plots, we analyse behaviour of the temperature in three different cases. In the top left plot, the graphic of the initial function is given. In the further plots, we compare the temperature function $u$ which is the solution of \eqref{RE-01} at $t=2, 6, 10$ for $\varepsilon=0.2$ in three cases. Case 1 is corresponding to the potential $q$ equal to zero. Case 2 is corresponding to the case when the potential $q$ is a $\delta$-function with the support at point $40$. Case 3 is corresponding to a $\delta^2$-like function potential with the support at point $40$.} \label{fig1}
\end{figure}

\begin{figure}[ht!]
\begin{minipage}[h]{0.47\linewidth}
\center{\includegraphics[scale=0.35]{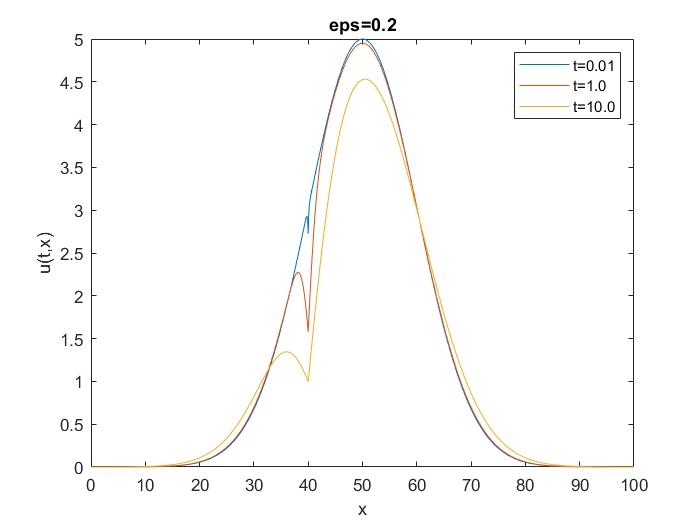}}
\end{minipage}
\hfill
\begin{minipage}[h]{0.47\linewidth}
\center{\includegraphics[scale=0.35]{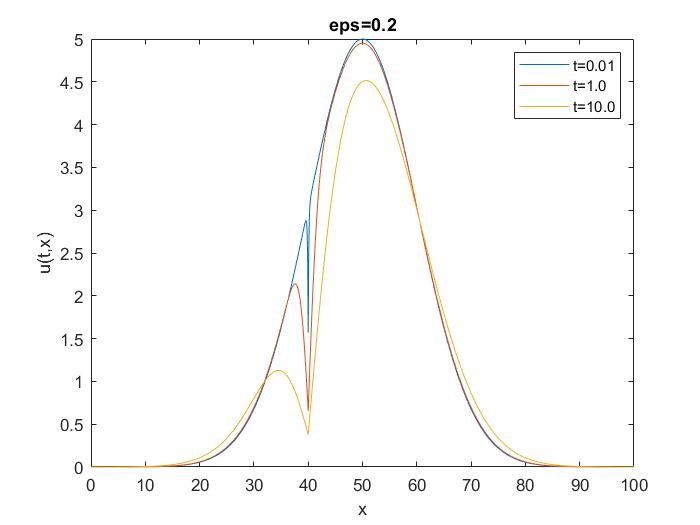}}
\end{minipage}
\caption{In these plots, we compare the temperature function $u$ at $t=0.01, 1.0, 10.0$ for $\varepsilon=0.2$ in the second and third cases: when the potential is $\delta$-like and $\delta^{2}$-like functions with the support at point $40$, respectively. The left picture is corresponding to the second case. The right picture is corresponding to the third case.}
\label{fig2}
\end{figure}

\begin{figure}[ht!]
\begin{minipage}[h]{0.30\linewidth}
\center{\includegraphics[scale=0.25]{./u0.jpg}}
\end{minipage}
\hfill
\begin{minipage}[h]{0.30\linewidth}
\center{\includegraphics[scale=0.25]{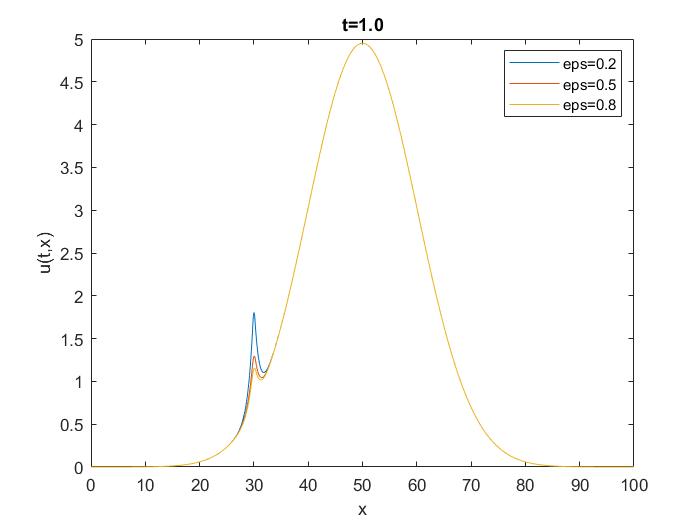}}
\end{minipage}
\hfill
\begin{minipage}[h]{0.30\linewidth}
\center{\includegraphics[scale=0.25]{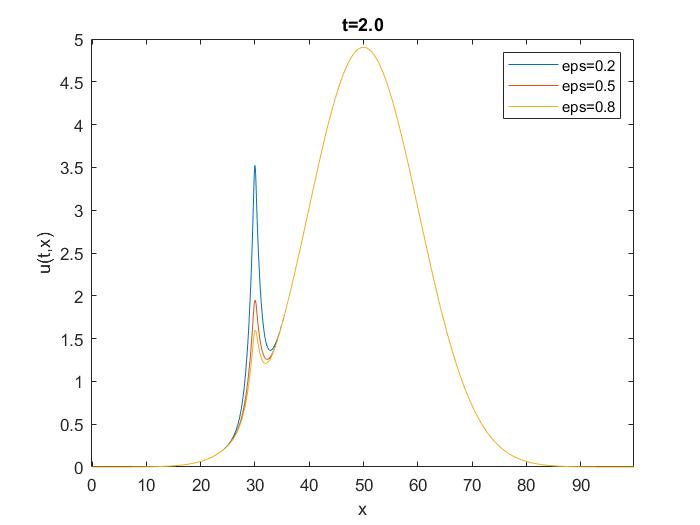}}
\end{minipage}
\hfill
\begin{minipage}[h]{0.30\linewidth}
\center{\includegraphics[scale=0.25]{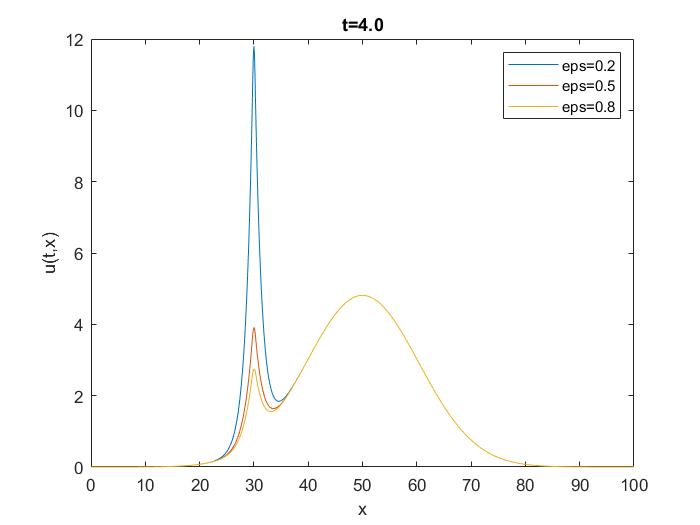}}
\end{minipage}
\hfill
\begin{minipage}[h]{0.30\linewidth}
\center{\includegraphics[scale=0.25]{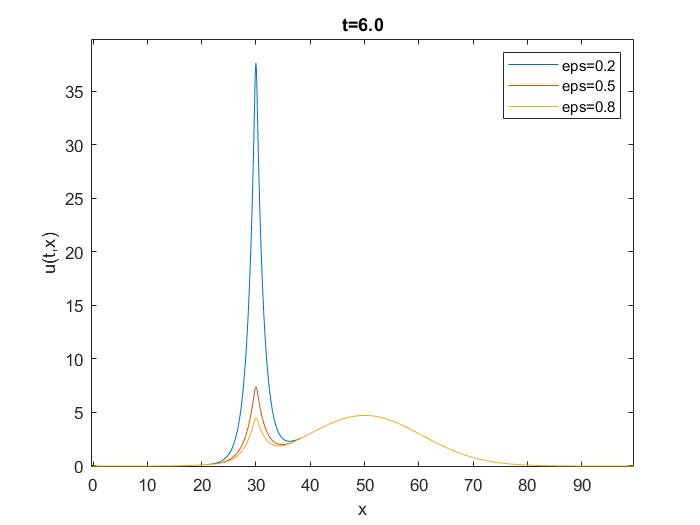}}
\end{minipage}
\hfill
\begin{minipage}[h]{0.30\linewidth}
\center{\includegraphics[scale=0.25]{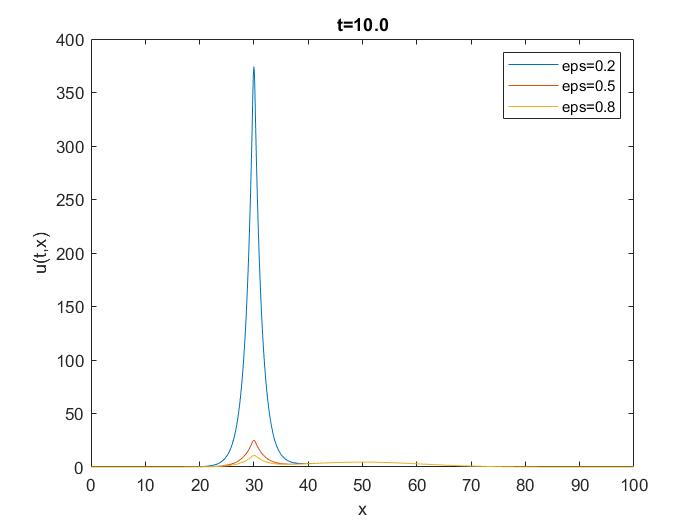}}
\end{minipage}
\caption{In these plots, we analyse behaviour of the solution of the heat equation \eqref{RE-01-Negative} with the negative potential. In the top left plot, the graphic of the temperature distribution at the initial time. In the further plots, we compare the temperature function $u$ at $t=1, 2, 4, 6, 10$ for $\varepsilon=0.8, 0.5, 0.2$. Here, the case of the potential with a $\delta$-like function behaviour with the support at point $30$ is considered.}
\label{fig3}
\end{figure}

\section{Numerical experiments}

In this Section, we do some numerical experiments. Let us analyse our problem by regularising a distributional potential $q(x)$ by a parameter $\varepsilon$. We define
$
q_\varepsilon (x):=(q\ast\varphi_\varepsilon)(x),
$
as the convolution with the mollifier
$\varphi_\varepsilon(x)=\frac{1}{\varepsilon} \varphi(x/\varepsilon),$
where
$$
\varphi(x)=
\begin{cases}
c \exp{\left(\frac{1}{x^{2}-1}\right)}, |x| < 1, \\
0, \,\,\,\,\,\,\,\,\,\,\,\,\,\,\,\,\,\,\,\,\,\,\,\,\,\,\,\,  |x|\geq 1,
\end{cases}
$$
with  $c \simeq 2.2523$ to have
$
\int\limits_{-\infty}^{\infty}  \varphi(x)dx=1.
$
Then, instead of \eqref{Equation} we consider the regularised problem
\begin{equation}\label{RE-01}
\partial_{t}u_{\varepsilon}(t,x)-\partial^{2}_{x} u_{\varepsilon}(t,x)+ q_{\varepsilon}(x) u_{\varepsilon}(t,x) =0, \; (t,x)\in[0,T]\times\mathbb R,
\end{equation}
with the initial data $u_\varepsilon(0,x)=u_0 (x)$, for all $x\in\mathbb R.$ Here, we put
\begin{equation}
\label{u_0}
u_0 (x)=
\begin{cases}
\exp{\left(\frac{1}{(x-50)^{2}-0.25}\right)}, \,\, |x-50| < 0.5, \\
0, \,\,\,\,\,\,\,\,\,\,\,\,\,\,\,\,\,\,\,\,\,\,\,\,\,\,\,\,\,\,\, \,\,\,\,\,\, \,\,\, |x-50| \geq 0.5.
\end{cases}
\end{equation}
Note that ${\rm supp }\, u_0\subset[49.5, 50.5]$.

In the non-negative potential case, for $q$ we consider the following cases, with $\delta$ denoting the standard Dirac's delta-distribution:
\begin{itemize}
  \item[Case 1:] $q(x)=0$ with $q_{\varepsilon}(x)=0$;
  \item[Case 2:] $q(x)=\delta(x-40)$ with $q_{\varepsilon}(x)=\varphi_\varepsilon(x-40)$;
  \item[Case 3:] $q(x)=\delta(x-40)\times\delta(x-40)$. Here, we understand $q_{\varepsilon}(x)$ as follows $q_{\varepsilon}(x)=\left(\varphi_\varepsilon(x-40)\right)^{2};$
\end{itemize}

In Figure \ref{fig1}, we study behaviour of the  temperature function $u$ which is the solution of \eqref{RE-01} at $t=2, 6, 10$ for $\varepsilon=0.2$ in three cases: the first case is corresponding to the potential $q$ equal to zero; the second case is corresponding to the case when the potential $q$ is a $\delta$-function with the support at point $40$; the third case is corresponding to a $\delta^2$-like function potential with the support at point $40$. By comparing these cases, we observe that in the second and in the third cases a place of the support of the $\delta$-function is cooling down faster rather that zero-potential case. This phenomena can be described as a "point cooling" or "laser cooling" effect.

In Figure \ref{fig2}, we compare the temperature function $u$ at $t=0.01, 1.0, 10.0$ for $\varepsilon=0.2$ in the second and third cases: when the potential is $\delta$-like and $\delta^{2}$-like functions with the supports at point $40$, respectively. The left picture is corresponding to the second case. The right picture is corresponding to the third case.

In Figures \ref{fig1} and \ref{fig2}, we analyse the equation \eqref{RE-01} with positive potentials. Now, in Figure \ref{fig3}, we study the following equation with negative potentials:
\begin{equation}\label{RE-01-Negative}
\partial_{t}u_{\varepsilon}(t,x)-\partial^{2}_{x} u_{\varepsilon}(t,x) - q_{\varepsilon}(x) u_{\varepsilon}(t,x) =0, \; (t,x)\in[0,T]\times\mathbb R,
\end{equation}
with the same initial data $u_0$ as in \eqref{u_0}. In these plots, we compare the temperature function $u$ at $t=1, 2, 4, 6, 10$ for $\varepsilon=0.8, 0.5, 0.2$ corresponding to the potential with a $\delta$-like function with the support at point $30$. Numerical simulations justify the theory developed in Section \ref{NP}. Moreover, we observe that the negative $\delta$-potential case a place of the support of the $\delta$-function is heating up. This phenomena can be described as a "point heating" or "laser heating" effect. Also, one observes that our numerical calculations prove the behaviour of the solution related to the parameter $\varepsilon$.

All numerical computations are made in C++ by using the sweep method. In above numerical simulations, we use the Matlab R2018b. For all simulations we take $\Delta t=0.2$, $\Delta x=0.01.$

\subsection{Conclusion} The analysis conducted in this article showed that numerical methods work well in situations where a rigorous mathematical formulation of the problem is difficult in the framework of the classical theory of distributions. The concept of very weak solutions eliminates this difficulty in the case of the terms with multiplication of distributions. In particular, in the potential heat equation case, we see that a delta-function potential helps to loose/increase energy in a less time, the latter causing a so-called "laser cooling/heating" effect in the positive/negative potential cases.

Numerical experiments have shown that the concept of very weak solutions is very suitable for numerical modelling. In addition, using the theory of very weak solutions, we can talk about the uniqueness of numerical solutions of differential equations with strongly singular coefficients in an appropriate sense.

\section*{Acknowledgement} This research was funded by the Science Committee of the Ministry of Education and Science of the Republic of Kazakhstan (Grant No. AP09058069) and by the FWO Odysseus 1 grant G.0H94.18N: Analysis and Partial Differential Equations. MR was supported in parts by the EPSRC Grant EP/R003025/2. AA was funded in parts by the SC MES RK Grant No. AP08052028. MS was supported by the Algerian Scholarship P.N.E. 2018/2019 during his visit to the University of Stuttgart and Ghent University. Also, Mohammed Sebih thanks Professor Jens Wirth and Professor Michael Ruzhansky for their warm hospitality.


\begin{thebibliography}{99}


\bibitem[AFP17]{AFP17} M.F. de Almeida, L.C.F. Ferreira, J.C. Precioso. \newblock On the Heat Equation with Nonlinearity and Singular Anisotropic Potential on the Boundary. \newblock {\it Potential Anal.}, (2017) 46, 589--608.


\bibitem[ARST21]{ARST21}
A. Altybay, M. Ruzhansky, M. Sebih, N. Tokmagambetov.
\newblock Fractional Klein-Gordon equation with singular mass.
\newblock {\it Chaos, Solitons and Fractals}, 143 (2021), 110579.

\bibitem[BG84a]{BG84a} P. Baras and J.A. Goldstein. \newblock Remark on the inverse square potential in quantum mechanics. \newblock {\it North-Holland Mathematics Studies}, Volume 92 (1984), 31--35.

\bibitem[BG84b]{BG84b} P. Baras and J.A. Goldstein. \newblock The heat equation with a singular potential. \newblock {\it Trans. Amer. Math. Soc.}, 284 (1984), 121--139.

\bibitem[Eva98]{Eva98} L.C. Evans. \newblock Partial Differential Equations. \newblock {\it American Mathematical Society}, 1998.

\bibitem[FJ98]{FJ98} F.G. Friedlander, M. Joshi. \newblock Introduction to the Theory of Distributions. \newblock {\it Cambridge University Press}, 1998.

\bibitem[FM15]{FM15} L.C.F. Ferreira, C.A.A.S. Mesquita. \newblock An approach without using Hardy inequality for the linear heat equation with singular potential. \newblock {\it Communications in Contemporary Mathematics}, (2015) 1550041.

\bibitem[GR15]{GR15} C. Garetto, M. Ruzhansky. \newblock Hyperbolic second order equations with non-regular time dependent coefficients. \newblock {\it Arch. Rational Mech. Anal.}, 217 (2015), no. 1, 113--154.

\bibitem[Gar20]{Gar20}
C. Garetto.
\newblock On the wave equation with multiplicities and space-dependent irregular coefficients.
\newblock {\it Preprint}, arXiv:2004.09657 (2020).

\bibitem[Gul02]{Gul02} A. Gulisashvili. \newblock On the heat equation with a time-dependent singular potential. \newblock {\it J. Functional Analysis}, 194 (2002), 17--52.

\bibitem[IKM19]{IKM19} K. Ishigea, Y. Kabeyab and A. Mukaia. \newblock Hot spots of solutions to the heat equation with inverse square potential. \newblock {\it App. Analysis}, 98 (2019) No. 10, 1843--1861.

\bibitem[IO19]{IO19} N. Iokua, T. Ogawab. \newblock Critical dissipative estimate for a heat semigroup with a quadratic singular potential and critical exponent for nonlinear heat equations. \newblock {\it J. Differential Equations}, 266 (2019) 2274--2293.

\bibitem[Sch54]{Sch54} L. Schwartz. \newblock Sur l impossibility de la multiplication des distributions. \newblock {\it C. R. Acad. Sci. Paris}, 239 (1954) 847--848.

\bibitem[Mar03]{Mar03} C. Marchi. \newblock The Cauchy problem for the heat equation with a singular potential. \newblock {\it Differential and Integral Equations}, 16, No 9 (2003) 1065--1081.

\bibitem[MRT19]{MRT19} J.C. Munoz, M. Ruzhansky and N. Tokmagambetov. \newblock Wave propagation with irregular dissipation and applications to acoustic problems and shallow water. \newblock {\it Journal de Math\'ematiques Pures et Appliqu\'ees}. Volume 123, March 2019, Pages 127--147.

\bibitem[MS10]{MS10} S.Moroz, R. Schmidt. \newblock Nonrelativistic inverse square potential, scale anomaly, and complex extension. \newblock {\it Annals of Physics}, 325 (2010) 491--513.

\bibitem[RT17a]{RT17a} M. Ruzhansky, N. Tokmagambetov. \newblock Very weak solutions of wave equation for Landau Hamiltonian with irregular electromagnetic field. \newblock {\it Lett. Math. Phys.}, 107 (2017) 591--618.

\bibitem[RT17b]{RT17b} M. Ruzhansky, N. Tokmagambetov. \newblock Wave equation for operators with discrete spectrum and irregular propagation speed. \newblock {\it Arch. Rational Mech. Anal.}, 226 (3) (2017) 1161--1207.

\bibitem[VZ00]{VZ00} J.L. Vazquez, E. Zuazua. \newblock The Hardy Inequality and the Asymptotic Behaviour of the Heat Equation with an Inverse-Square Potential. \newblock {\it J. Functional Analysis},  173, 103--153 (2000).




\end{thebibliography}
\end{document}